\documentclass[11pt]{amsart}

\usepackage {enumerate}
\usepackage[english]{babel}
\usepackage{setspace}
\usepackage{caption}
\usepackage{amssymb}
\usepackage{amsfonts}
\usepackage{amsmath}
\usepackage{graphicx}

\usepackage{color}
\usepackage[margin=1.5in]{geometry}
\allowdisplaybreaks

\newtheorem{thm}{Theorem}[section]
\newtheorem{lm}{Lemma}[section]
\newtheorem{defn}{Definition}[section]
\newtheorem{pro}{Proposition}[section]
\newtheorem{re}{Remark}[section]
\newtheorem{cor}{Corollary}[section]

\newcommand{\R}{\mathbb{R}}
\usepackage{esint}
\usepackage{tikz}
\usetikzlibrary{decorations.pathreplacing}

\begin{document}
\title[]
{Isoperimetry for asymptotically flat 3-manifolds with positive ADM mass }

\author{Haobin Yu}
\address[Haobin Yu]{Department of Mathematics, Hangzhou Normal University, Hangzhou, 311121, P.\ R.\ China}
 \email{robin1055@126.com}
\date{}

\begin{abstract}
Let  $(M^3, g)$ be an asymptotically flat 3-manifold with  positive ADM mass.
In this paper, we show that  each leaf of the canonical foliation is the unique isoperimetric surface for the volume it encloses.  Our proof is based on the "fill-in" argument and sharp isoperimetric inequality
on asymptotically flat 3-manifold with nonnegative scalar curvature.

\end{abstract}

\maketitle
\markboth{Yu Haobin}{Isoperimetry for AF 3-manifolds with positive mass}

\section{Introduction}
A three manifold $(M, g)$ is said to be asymptotically flat if there are a compact subset $K\subseteq  M$ and a chart
\begin{equation}\label{chart:infinity}
M \setminus K\cong \R^3 \setminus \overline{B_{\frac{1}{2}}(0)}
\end{equation}
so that the components of the metric tensor have the form
\[
g_{ij} = \delta_{ij} + \sigma_{ij},\]
where
\begin{equation}\label{tau}
|x|^{\alpha}|\partial^{\alpha}\sigma_{ij}(x)|=O(|x|^{-\tau}),
\ \ \ \text{as}\ \ \ \ |x|\rightarrow \infty
\end{equation}
for some $\tau > 1/2$ and all multi-indices $\alpha$ with $|\alpha|=0,1,2$.
We also require that the scalar curvature of $(M, g)$ is integrable. 
The ADM-mass (after Arnowitt, Deser and Misner \cite{ADM:1961}) of such an asymptotically flat manifold $(M, g)$ is given by 
\[
m_{ADM} =  \lim_{\rho \to \infty } \frac{1}{16 \pi \rho} \int_{\{|x| = \rho\}} \sum_{i, j = 1}^3 \left( \partial_i g_{ij} - \partial_j g_{ii} \right) x^j
\]
where integration is with respect to the Euclidean metric. 

In their seminal paper \cite{HY}, Huisken and Yau proved that if 
$(M, g)$  is $C^4$-asymptotic to Schwarzschild of mass $m>0$, then out of some compact set, $M$  can be foliated by a family of strictly volume preserving stable constant mean curvature spheres $\{\Sigma_H\}_{H\leq H_0}$. 
Moreover, the leaves of this foliation are the unique volume preserving stable CMC spheres of their mean curvature within a large class of surfaces. Their uniqueness result was later strengthened by Qing and Tian\cite{QT}.
Various extensions of these results that allow for weaker asymptotic conditions
 have been proven in \cite{Huang2010,Ma2011,Ma2016}.
The following  optimal existence and uniqueness results for general asymptotically flat 3-manifolds
 was established by Nerz in \cite{Nerz2015}.
\begin{thm}\label{thm:existence:CMC}
Let $(M, g)$ be a complete Riemannian $3$-manifold that is asymptotically flat at rate $\tau > 1/2$ and which has  positive mass.  Suppose the scalar curvature of $(M, g)$ is nonnegative or satisfies $R(g)=O(|x|^{-\frac{5}{2}-\tau})$.
Then for some compact $L$,
$M^3\setminus L$ can be foliated by stable CMC spheres $\{\Sigma_{\sigma}\}_{\sigma>\sigma_*}$ with
 $\frac{2}{\sigma}\ll 1$ being the mean curvature of $\Sigma_{\sigma}$. Moreover, any large stable CMC sphere 
with mean curvature $\frac{2}{\sigma}$ and which is geometrically close to $S_{\sigma}$ must coincides with
$\Sigma_{\sigma}$.
\end{thm}
\begin{re}
In \cite{Nerz2015}, Nerz showed that the decay assumptions are optimal and cannot be weakened to guarantee 
the existence of canonical  foliation.
\end{re}
It's a very natural question to ask if the leaves of the canonical foliation
 are isoperimetric surfaces for the volume they enclose, i.e.,
given $V\gg1$, whether the  isoperimetric profile $A(V)$ can be achieved by leaves $\{\Sigma_V\}$ of the canonical foliation. Here $A:[0, \infty)\rightarrow [0, \infty)$ is defined by
\begin{align}\label{inf}
A(V)=\inf\{&\mathcal{H}^2(\partial ^*\Omega):\Omega\subset M \text{
is a compact region and}\ \ 
\mathcal{L}^3(\Omega)=V
\},
\end{align}
where $\mathcal{ H}^2$ is 2-dimensional Hausdorff measure for the reduced boundary of $\Omega$,
and $\mathcal{L}^3(\Omega)$ is the Lebesgue measure of $\Omega$ with respect to metric $g$.

In the asymptotically Schwarzschild setting, the study of isoperimetric structure on asymptotically flat Riemannian 3-manifolds $(M^3, g)$ may  date back to Bray's work. In \cite{Bray}, Bray showed that the isoperimetric surfaces of
spatial Schwarzschild manifold are exactly round centered spheres. He
deduced that  if $(M^3, g)$ is the compact perturbations of the exact Schwarzschild metric
then the large isoperimetric surfaces are also round centered spheres.
By isoperimetric  technique, Bray gave a proof of Penrose inequality using positive mass theorem.
He conjectured that the volume-preserving stable constant mean curvature spheres
constructed by Huisken-Yau \cite{HY}
are isoperimetric surfaces. Building on Bray's volume comparison, 
Eichmair-Metzger \cite{EM:2013:1,EM:2013:2} obtained
global uniqueness of large solutions of the isoperimetric problem in any dimension for
$(M, g)$ asymptotic to Schwarzschild with mass $m > 0$
and  they gave a confirm answer to Bray's conjecture.

To study the isoperemetric properties of  3-manifolds with general asymptotics, Huisken\cite{Huisken2006}  introduced the the concepts of isoperimetric mass and quasilocal isoperimetric mass
which only require very low regularity.
\begin{defn}(\text{Huisken}) Let $(M^3, g)$ be a $C^0$-asymptotically flat manifold
 and $\Omega$ be a smooth bounded domain.
The quasilocal isoperimetric mass of $\Omega$ is
\[
m_{iso}(\Omega)=\frac{2}{\mathcal{H}^2(\partial\Omega)}\Big(\mathcal{L}^3(\Omega)
-\frac{1}{6\sqrt{\pi}}\mathcal{H}^2(\partial\Omega)^{\frac{3}{2}}\Big).
\]
The isoperimetric mass of $(M^3, g)$ is defined by
\[
m_{iso}(M, g)=\sup_{\{\Omega_i\}_{i=1}^{\infty}}(\limsup_{i\rightarrow \infty}m_{iso}(\Omega_i)),
\]
where $\{\Omega_i\}$ is  an exhaustion of $(M, g)$.
\end{defn}
Subsequent to the work of Huisken, Fan-Miao-Shi-Tam \cite{FST} observed that the "lim sup" in  Huisken's
definition  recovers the ADM-mass of the initial data set when evaluated
along exhaustions by concentric coordinate balls in an asymptotic coordinate system.
Hence, $m_{ADM}(M,g)\leq m_{iso}(M, g)$. 
The following result was proposed by  Huisken\cite{Huisken2006,Huisken2009}
and proven by Jauregui-Lee\cite{JL2017}. 
An alternative proof was given by the author in joint work with Chodosh-Eichmair-Shi\cite{CESY}. (Both approaches also use an important insight by Fan-Miao-Shi-Tam\cite{FST}.)
\begin{thm}
Let $(M,g)$ be an asymptotically flat Riemannian $3$-manifold at decay rate $\tau>\frac{1}{2}$ and
which has non-negative scalar curvature. Then
\[
m_{ADM}(M,g)=m_{iso}(M, g).
\]
\end{thm}
Let $(M,g)$ be an asymptotically flat Riemannian $3$-manifold with non-negative scalar curvature
and $\Omega \subset M$ be a compact region.  An immediate consequence of the theorem above is
\begin {thm} [Sharp isoperimetric inequality] \label{cor:isoperimetricinequality} 
\begin{align} \label{eqn:isoperimetricinequality}
V (\Omega) \leq \frac{A(\partial \Omega)^{3/2}}{6 \sqrt \pi} +  \frac{m_{ADM}}{2} A (\partial \Omega) + o(1) A(\partial \Omega)
\end{align}
as $V(\Omega) \to \infty$.
\end {thm}
In a notable paper, Shi \cite{Shi2016} established the isoperimetric inequality on 
asymptotically flat 3-manifolds with non-negative scalar curvature,
based on which Carlotto, Chodosh and Eichmair\cite{CCE}  showed that for any $V>0$, 
there always exists a smooth isoperimetric region $\Omega$ with $Vol(\Omega)=V$.
In a jointed work with Chodosh, Eichmair and Shi\cite{CESY},  
 we gave a complete characterization of  isoperimetric structure in large scale for asymptotically flat Riemannian $3$-manifold with non-negative scalar curvature.
\begin{thm}\label{CESY}
Let $(M, g)$ be a complete Riemannian $3$-manifold that is asymptotically flat at rate $\tau > 1/2$ and which has non-negative scalar curvature and positive mass. There is $V_0 > 0$ with the following property. Let $V \geq V_0$. There is a unique isoperimetic region $\Omega_V $ with $V(\Omega_V)=V$ whose boundary consists of the horizon $\partial M$ and a leaf of the canonical foliation of the end of $M$.
\end{thm}
\begin{re}
In \cite{CESY}, a central step is to establish the effective volume comparison for 
large constant mean curvature surfaces as Eichmair-Metzger did in \cite{EM:2013:1,EM:2013:2}. 
To this end, we use the sharp isoperimetric inequality and a monotonicity formula under mean curvature flow, which was discovered  by Huisken \cite{Huisken2006,Huisken2009} and generalized by Jauregui-Lee \cite{JL2017} for modified mean curvature flow.
\end{re}
Without  the nonnegative scalar curvature assumption, the general existence of large isoperimetric regions 
was established by Carlotto, Chodosh and Eichmair in \cite{CCE}:
\begin{thm}
Let $(M^3, g)$ be an asymptotically flat Riemannian 3-manifold with horizon boundary, integrable scalar curvature, and positive ADM-mass.
 For all $V > 0$ sufficiently large there is a smooth isoperimetric region of volume $V$.
\end{thm}
The main theorem of this paper can be stated as follows:
\begin{thm}\label{mainthm}
Let $(M^3,g)$ be an asymptotically flat Riemannian 3-manifold with positive ADM mass $m>0$. 
Suppose the scalar curvature of $(M^3, g)$ is nonnegative or satisfies 
$R(g)=O(|x|^{-\frac{5}{2}-\tau})$.
Then there exists some $V_0>0$ such that for any $V>V_0$ there is a unique isoperimetric region $\Omega_V$ whose 
boundary is a  leaf in the canonical foliation $\{\Sigma_{\sigma}\} _{\sigma>\sigma_*}$.
\end{thm}
As a corollary, we immediately have
\begin{cor}
Assume as in the theorem above. 
Then
\[
m_{ADM}(M,g)=m_{iso}(M, g).
\]
\end{cor}

One ingredient in our proof is the "fill-in" argument and 
 we are partially inspired by the recent work \cite{SWWZ}. To the author's best knowledge, 
this  argument is completely new in dealing with  isoperimetric problems on asymptotically flat manifolds. 
Let $\Sigma_{\sigma}$ be any leaf  of the canonical foliation, cut the domain enclosed by the leaf and fill $\Sigma_{\sigma}$ with a
suitable metric in a canonical way.  As the metric we construct has corners,
we need to smooth the metric and then take conformal deformation
to get a family of asymptotically flat metrics with nonnegative scalar curvature. 
The ADM mass of deformed metrics are strictly less than the mass of initial metric 
if the leaf $\Sigma_{\sigma}$ we choose is far away enough.
Building on the sharp isoperimetric inequality,
 we can show that  the isoperimetric regions must look like
Euclidean balls $B_1(0)$ when scaled by their volume.

We remark here that recent breakthrough was made by Chodosh and Eichmair \cite{CE:2017:1,CE:2017:2}.
They established the optimal, global result for stable constant mean curvature spheres in initial data asymptotic to Schwarzschild with nonnegative scalar curvature.
Finally, we mention some recent progress in the asymptotically hyperbolic setting.
 Chodosh \cite{Chodosh:2016} has shown that large isoperimetric surfaces are centered coordinate spheres in the special case where the metric is isometric to Schwarzschild-anti-de Sitter outside of a compact set.
Under the assumption that the manifold $(M^3, g)$
is asymptotic to Schwarzschild-anti-de Sitter with scalar curvature $R\geq -6$, 
Chodosh, Eichmair, Shi and Zhu \cite{CESZ} showed that
the leaves of the canonical foliation constructed by Rigger \cite{Rigger:2004} are unique isoperimetric surfaces for the volume they enclose. In their case, the scalar curvature assumption is necessary.

The remains of  papers are organized as follows:
In Section $2$, we construct a family of metrics $\{g_{\sigma}\}$ which coincide with $g$ outside of $\Sigma_{\sigma}$ 
and have nice behaviour in the domain enclosed by $\Sigma_{\sigma}$.
In Section $3$,  we get a family of asymptotically flat metrics with nonnegative scalar curvature
by deforming $\{g_{\sigma}\}$  as Miao did in \cite{Miao2002}
and show that the mass of deformed metrics  can be strictly less than $m$. 
We give the proof of the main theorem in the last section, \\

\noindent {\bf Acknowledgments.} The author would like to express his gratitude to Professor  Shi Yuguang for his constant encouragement. We sincerely thank Professor Michael Eichmair for his
valuable suggestions. We also thank Otis Chodosh and Wang wenlong for helpful discussions.  

\section{Gluing the metric}
 Let $\omega_{\sigma}$  be the induced metric on $\Sigma_{\sigma}$.  
 Set $\omega_{\sigma}=e^{2u_{\sigma}}\sigma^2g_*$,  here $g_*$ is some round metric on $S^2$ with area $4\pi$ and $u_{\sigma}$ is a function defined on $S^2$.
Then Nerz\cite{Nerz2015} showed that for some fixed $\alpha\in(0,1)$, it holds
\begin{equation}
||\sigma^{-2}\omega_{\sigma}-g_*||_{C^{2,\alpha}(S^2)} \leq C\sigma^{-\tau},\ \ \text{for}\ \  \sigma\gg1.
\end{equation}
Here and in the following, we always use $C$ to denote  universal constants
 depending only on $(M^3, g)$ which may vary from line to line. Then we have
\begin{equation}\label{estimate:u}
|u_{\sigma}|_{C^{2,\alpha}(S^2)}\leq C\sigma^{-\tau},
\ \ \text{for}\ \  \sigma\gg1.
\end{equation}
We need the following result obtained by Mantoulidis and Schoen in \cite{MS2015}.
\begin{lm}\label{extension}
Let $\omega=e^{2u_{\sigma}}\sigma^2g_*$  be as above. 
Then there exists a smooth path of metrics $t \mapsto \omega(t)$ such that 
\[
\omega(0)=\omega,\ \ \omega(\frac{\sigma}{2})\ \  \text{round},\ \ \frac{d}{dt}dA_{\omega(t)}\equiv0,\ \ \text{for all}\ \ t\in[0,\frac{\sigma}{2}],
\]
where $dA_{\omega(t)}$ denotes the area form for a metric $\omega(t)$.
\end{lm}
\begin{proof}
The argument here follows are from \cite{MS2015}. 
Consider
\[
\tilde{\omega}(t)=e^{2u_{\sigma}(1-\frac{2t}{\sigma})+2a(t)}\sigma^2 g_*
\]
with $a(t)$ chosen so that $a(0)=0$
and 
\[
a'(t)=\frac{2}{\sigma}\fint_{S^2}u_{\sigma}dA_{\tilde{\omega}(t)}=O(\sigma^{-1-\tau}).
\]
Consider the following equation 
\begin{equation}\label{possion:equ}
\Delta_{\tilde{\omega}(t)}\psi(t,\cdot)=\frac{4u_{\sigma}}{\sigma}-2a'(t).
\end{equation}
It is solvable  since the integral of the righthand term vanishes. Then the standard elliptic estimate gives
\begin{equation}\label{est:ellip}
|\psi(t,\cdot)|_{C^2(S^2, \tilde{\omega})}\leq C\sigma^{-1-\tau},\ \ \ \ \ \text{for}\ \ \ t\in[0,\frac{\sigma}{2}].
\end{equation}
Take $X_t=\nabla^{\tilde{\omega}(t)}\psi(t,\cdot)$ and let $\phi_t$ be the one-parameter diffeomorphism group generated by $X_t$.
Consider $\omega(t)=\phi_t^*\tilde{\omega}(t)$. Then
\begin{align*}
\frac{d}{dt}dA_{\omega(t)}
=&\frac{d}{dt}\phi^*_tdA_{\tilde{\omega}(t)}=\phi_t^*\Big[\frac{d}{dt}dA_{\tilde{\omega}(t)}
+L_{\dot{\phi_t}}dA_{\tilde{\omega}(t)}\Big]\\
=&\phi^*_t\Big[\frac{1}{2}\text{tr}_{\tilde{\omega}(t)}\dot{\tilde{\omega}}(t)dA_{\tilde{\omega}(t)}
+\text{div}_{\tilde{\omega}(t)}\dot{\phi_t}dA_{\tilde{\omega}(t)}\Big]\\
=&\phi^*_t\Big[2a'(t)-\frac{4u_{\sigma}}{\sigma}+\Delta_{\tilde{\omega}(t)}\psi(t,\cdot)\Big]
dA_{\tilde{\omega}(t)}=0,
\end{align*}
where $L$ denotes the Lie derivative on $S^2$. Hence, we complete the proof.
\end{proof}
Now we define a family of metrics on $\Sigma_{\sigma}\times[0, \frac{\sigma}{2}]$ by  
\[
\gamma=f(t)\omega(t)+dt^2=
\Big(1-\frac{t}{\sigma}\Big)^2\omega(t)+dt^2.
\]
 Let $\Omega_{\sigma}$ be the domain enclosed by 
$\Sigma_{\sigma}$. Denote the surface
 $\Sigma_{\sigma}\times\{\frac{\sigma}{2}\}$ by $\Sigma_{\sigma}'$ and fill  $\Sigma_{\sigma}'$ with Euclidean ball $(g_E, \Omega_{\sigma}')$ such that
$g_E\big{|}_{\Sigma_{\sigma}'}=\gamma\big{|}_{\Sigma_{\sigma}'}$.
 We define a family of asymptotically flat metrics $\{g_{\sigma}\}$ with corners as follows:
\begin{equation}
g_{\sigma}=\left\{
\begin{aligned}
&g_E\ \  x\in\Omega_{\sigma}'\\
&\gamma\ \  \ x\in\Omega_{\sigma}\setminus \overline{\Omega_{\sigma}'}\\
&g\ \ \ x\in M^3\setminus\overline{\Omega_{\sigma}}
\end{aligned}
\right.
\end{equation}
\begin{lm}\label{PRE}
The scalar curvature of $\gamma$ satisfies 
\[
R_{\gamma}=O(\sigma^{-2-\tau})
\ \ \text{for}\ \ \sigma\gg1.
\]
\end{lm}
\begin{proof}
Set $h(t)=f(t)\omega(t)$.  Then
differentiating with respect to $t$ gives
\[
\dot{h}=f'\omega+f\dot{\omega}.
\]
Recall that $tr_{\omega}\dot{\omega}=0$. It follows
\[
tr_h\dot{h}=\frac{2f'}{f} \ \ \ \text{and}\ \ \ |\dot{h}|_h^2=\frac{2f'^2}{f^2}+|\dot{\omega}|^2_{\omega}.
\]
Differentiating and tracing again,
\[
\ddot{h}=f''\omega+2f'\dot{\omega}+f\ddot{\omega}\ \ \text{and}\ \ tr_h\ddot{h}=\frac{2f''}{f}+tr_{\omega}\ddot{\omega}.
\]
Using $tr_{\omega}\dot{\omega}=0$, we get
\[
tr_{\omega}\ddot{\omega}=|\dot{\omega}|^2_{\omega}.
\]
The scalar curvature of $\gamma$ is given by
\begin{align}\label{scalar:curv1}
R_{\gamma}=&2K_h-tr_h\ddot{h}-\frac{1}{4}(tr_h\dot{h})^2+\frac{3}{4}|\dot{h}|_h^2\nonumber\\
=&2K_h-\frac{2f''}{f}-|\dot{\omega}|^2_{\omega}-\frac{f'^2}{f^2}+\frac{3}{4}(\frac{2f'^2}{f^2}+|\dot{\omega}|^2_{\omega})\nonumber\\
=&2K_h+\frac{f'^2}{2f^2}-\frac{2f''}{f}-\frac{1}{4}|\dot{\omega}|_{\omega}^2\nonumber\\
=&2K_h-\frac{2}{f\sigma^2}-\frac{1}{4}|\dot{\omega}|_{\omega}^2.
\end{align}
By (\ref{estimate:u}), the Gauss curvature of $h $ can be estimated by  
\begin{align}\label{Gauss:curv}
K_h=&f^{-1}\phi^*_t\Big[e^{-2u_{\sigma}(1-\frac{2t}{\sigma})-2a(t)}\sigma^{-2}
\Big(1-2(1-\frac{2t}{\sigma})\Delta_{S^2}u_{\sigma}\Big)\Big]\nonumber\\
=& \frac{1}{f\sigma^{2}}+O(\sigma^{-2-\tau}).
\end{align}
Note that
\[
\dot{\omega}=\phi_t^*\Big[e^{2u_{\sigma}(1-\frac{2t}{\sigma})+2a(t)}
\sigma^2(2a'(t)-\frac{4u_{\sigma}}{\sigma}) g_*
+\nabla^{2}_{\tilde{\omega}}\psi\Big].
\]
Then it follows from (\ref{est:ellip}) that
\begin{equation}\label{error:term1}
|\dot{\omega}|_{\omega}^2=O( \sigma^{-2-2\tau}).
\end{equation}
Substituting (\ref{Gauss:curv}) and (\ref{error:term1}) into (\ref{scalar:curv1}) gives
the desired estimate.
\end{proof}
\begin{lm}\label{mean:cur}
Let $\{-\frac{\partial}{\partial t}\}$ be the outer normal vector field on 
$\Sigma_{\sigma}\times[0,\frac{\sigma}{2}]$. Then
for $\sigma\gg 1$, we have
\[
 H(\Sigma_{\sigma},\gamma)=\frac{2}{\sigma}=H(\Sigma_{\sigma},g)\ \ \text{and}
 \ \ H(\Sigma_{\sigma}',g_E)=\frac{4}{R_{\sigma}}\geq\frac{4}{\sigma}= H(\Sigma_{\sigma}',\gamma),
\]
where $R_{\sigma}$ is the constant such that $A(\Sigma_{\sigma}, g)=4\pi R^2_{\sigma}$. 
\end{lm}
\begin{proof}A direct computation shows that
\[
H(\Sigma_{\sigma}, \gamma)=-\frac{f'(0)}{f(0)}=\frac{2}{\sigma}=H(\Sigma_{\sigma},g).
\]
 The mean curvature of $\Sigma_{\sigma}'$ with respect to metric $\gamma$ and $g_E$ are
 repectively given by
\begin{align}
H(\Sigma_{\sigma}',\gamma)=&\frac{-f'(\frac{\sigma}{2})}{f(\frac{\sigma}{2})}=\frac{4}{\sigma}\ \ \ \text{and}\ \ 
\ \ \ H(\Sigma_{\sigma}',g_E)=\frac{2}{\sqrt{f(\frac{\sigma}{2})}R_{\sigma}}=\frac{4}{R_{\sigma}}.
\end{align}
Note the Hawking mass of $\{\Sigma_{\sigma}\}$ satisfy
\[
m_H(\Sigma_{\sigma})=\sqrt{\frac{A(\Sigma_{\sigma},g)}{16\pi}}
\Big(1-\frac{A(\Sigma_{\sigma}, g)H^2(\Sigma_{\sigma},g)}{16\pi}\Big) \rightarrow m,
\ \ \text{as} \ \ \sigma\rightarrow \infty.
\]
Then we have
\begin{equation}\label{est:R}
1-\frac{R^2_{\sigma}}{\sigma^2}\geq\frac{m}{\sigma}\ \ \text{for}\ \ \sigma\gg1.
\end{equation}
This finishes the proof.
\end{proof}
By our construction,  we have
\begin{cor}\label{sob:inequ}
For $\sigma\gg 1$,  the Sobolev Constant of $\{g_{\sigma}\}$
is controlled by $C$ depending only on $(M^3, g)$.
\end{cor}

\section{Smoothing $\{g_{\sigma}\}$ and conformal deformations}

 In this section, we establish some estimates for certain conformal deformation equations. To begin with, we smooth the metric $g_{\sigma}$ across $\Sigma_{\sigma}$ and $\Sigma'_{\sigma}$ as Miao did in \cite{Miao2002}. Namely, we have the following proposition. 
\begin{pro}\label{smoothing the metric g}
There exists a family of $C^2$ metrics $\{g_{\sigma,\delta}\}_{0<\delta\leq\delta_*}$ on $\R^3$
 so that $g_{\sigma,\delta}$ is uniformly close to $g_{\sigma}$ on $\R^3$, $g_{\sigma,\delta}=g_{\sigma}$ outside $(\Sigma_{\sigma}\cup\Sigma'_{\sigma})\times(-\delta, \delta)$ (Gaussian  coordinates) and the scalar curvature of $g_{\sigma,\delta}$ satisfies
\begin{equation}
R_{\sigma,\delta}(z, t)=
\left\{
\begin{aligned}
& O(1)\qquad \qquad \qquad \qquad\ \  \text{for}\  (z, t)\in 
(\Sigma_{\sigma}\cup\Sigma'_{\sigma})\times\left\{\delta^2<|t|<\delta\right\},\\
&O(1)+\frac{H(z,\gamma)-H(z,g)}{\delta^2}\phi\left(\frac{t}{\delta^2}\right),\ \ \ \ \ \text{for }(z, t)\in \Sigma_{\sigma}\times[-\delta^2, \delta^2],\nonumber\\
&O(1)+\frac{H(z,g_E)-H(z,\gamma)}{\delta^2}\phi\left(\frac{t}{\delta^2}\right),\ \ \ \text{for }(z, t)\in \Sigma'_{\sigma}\times[-\delta^2, \delta^2],\nonumber
\end{aligned}
\right.
\end{equation}
where $O(1)$ represents quantities bounded by constants depending only on $g$, but not on $\delta$ or $\sigma$, and $\phi\in C^{\infty}_c([-1, 1])$ is a standard mollifier satisfying
$0\leq\phi\leq 1$, $\phi\equiv 1$ in $[-\frac{1}{3},\frac{1}{3}]$, and  $\int^1_{-1}\phi=1$.
\end{pro}
Then by Lemma \ref{mean:cur}, we have
\begin{equation}
R_{\sigma,\delta}=
\left\{
\begin{aligned}
& O(|x|^{-2-\tau}), \qquad\ \text{in}\ \ \Sigma_{\sigma}\times[\delta,\frac{\sigma}{2}-\delta]
\ \ \text{or}\ \ \text{outside}\ \ \Sigma_{\sigma}\times\{-\delta\}\\
&O(1)+\frac{4}{\delta^2}(\frac{1}{R_{\sigma}}-\frac{1}{\sigma})\phi\left(\frac{t}{\delta^2}\right), 
\quad \qquad\qquad  \  \text{in}\ \  \ \Sigma_{\sigma}'\times[-\delta^2,\delta^2]\\
& O(1),\qquad\  \ \ \ \qquad\qquad \big(\Sigma_{\sigma}\times[-\delta,\delta]\big)
\cup(\Sigma'_{\sigma}\times\{\delta^2<|t|<\delta\})\\
&0 \ \qquad \ \qquad \qquad\qquad  \qquad\qquad\qquad\qquad \qquad\qquad \text{otherwise}
\end{aligned}
\right.
\end{equation}
We choose some $C^2$ function $f_{\sigma,\delta}$ satisfying $f_{\sigma,\delta}=\frac{R_{\sigma,\delta}}{8}$ outside $\Sigma'_{\sigma}\times[-\delta,
\delta]$ and 
\[
-C_0\leq f_{\sigma,\delta}\leq\frac{R_{\sigma,\delta}}{8},\ \ \text{for}\ \ \ x\in\Sigma'_{\sigma}\times[-\delta,\delta].
\]
for some uniform $C_0>0$. Then
\[
\int |f_{\sigma,\delta}|^{\frac{3}{2}}dg_{\sigma,\delta}\leq 
C(\sigma^2\delta)^{\frac{3}{2}}+\frac{C}{\sigma}.
\]
Thus, for $\sigma\gg1$ and $\delta<\frac{1}{C\sigma^3}$, we have
\begin{align}\label{integral:estimate}
\int |f_{\sigma,\delta}|^{\frac{3}{2}}\leq\frac{C}{\sigma}\rightarrow 0 \ \ \text{as}\ \ \sigma\rightarrow \infty.
\end{align}
Consider the following equation
\begin{equation}\label{conformal:equ}
\left\{
\begin{aligned}
\Delta_{g_{\sigma, \delta}}u_{\sigma,\delta}-f_{\sigma,\delta}u_{\sigma,\delta}=0\\
u_{\sigma,\delta}(\infty)=\lim_{x\rightarrow \infty}u_{\sigma,\delta}(x)=1
\end{aligned}
\right.
\end{equation}
The solvability of equation (\ref{conformal:equ}) is guaranteed by the following lemma due to Schoen-Yau\cite{SY:1979}.
\begin{lm}\label{existence}
Let $(N,g_{\scriptscriptstyle N})$ be an asymptotically flat $3$-manifold
 and $h$ be a function that has the same decay rate at $\infty$ as $R(g_{\scriptscriptstyle N})$.
Then there exists a number $\epsilon_{\scriptscriptstyle N}>0$ depending only on the $C^0$ norm of $g_{\scriptscriptstyle N}$ and the decay rate of $ g_{\scriptscriptstyle N}$, $\partial g_{\scriptscriptstyle N}$ and $\partial^2 g_{\scriptscriptstyle N}$ at $\infty$ so that if
\begin{equation}
  \left(\int_{N}|h_-|^{\frac{3}{2}}\,d\mu_{g_{\scriptscriptstyle N}}\right)^{\frac{2}{3}}<\epsilon_{\scriptscriptstyle N},
\end{equation}
 then
\begin{equation*}
\left\{
\begin{aligned}
\Delta_{g_{\scriptscriptstyle N}}u-hu&=0\quad\text{in}\ N,\\
u&\rightarrow 1 \ \ \, \mbox{at}\  \infty.
\end{aligned}
\right.
\end{equation*}
has a $C^2$ positive solution $u$ that
\begin{equation*}
u(x)=1+\frac{A}{|x|}+B
\end{equation*}
for some constant $A$ and some function $B$, where $B=O(|x|^{-2})$ and $\partial B=O(|x|^{-3})$.
\end{lm}
Set $\tilde{g}_{\sigma,\delta}=u_{\sigma,\delta}^4g_{\sigma,\delta}$. 
Then $\{g_{\sigma,\delta}\}$ is a family of asymptotically flat metrics with nonnegative scalar curvature.
Let $v_{\sigma,\delta}=u_{\sigma,\delta}-1$.
\begin{lm}\label{C^0estimate}
For $\sigma\gg1$ and $\delta\ll \frac{1}{\sigma^3}$, it holds
\[
|v_{\sigma,\delta}|(x)\leq C\sigma^{-\frac{1}{2}},\ \ \ \text{for}\ \ |x|\geq\frac{\sigma}{2}.
\]
\end{lm}

\begin{proof}
We divid our proof into two steps.

\textbf{Step1}: We aim to get the $L^6$ estimate of $v_{\sigma,\delta}$. By (\ref{conformal:equ}),
\begin{equation}\label{eqn2:conformal}
\Delta_{g_{\sigma,\delta}}v_{\sigma,\delta}-f_{\sigma,\delta}v_{\sigma,\delta}
=f_{\sigma,\delta}.
\end{equation}
Multiplying the equation above with $v_{\sigma,\delta}$ and integrating on $\\R^3$ give
\[
\int (v_{\sigma,\delta}\Delta_{g_{\sigma,\delta}}v_{\sigma,\delta}+
\int f_{\sigma,\delta}v_{\sigma,\delta}^2)dg_{\sigma,\delta}
=\int f_{\sigma,\delta}v_{\sigma,\delta}dg_{\sigma,\delta}.
\]
Integrating by parts and using Holder Inequality, we have that
\begin{align}
\int|\nabla_{g_{\sigma,\delta}}v_{\sigma,\delta}|^2dg_{\sigma,\delta}
\leq&
\Big(\int |f_{\sigma,\delta}|^{\frac{3}{2}}\Big)^{\frac{2}{3}}
\Big(\int v_{\sigma,\delta}^6dg_{\sigma,\delta}\Big)^{\frac{1}{3}}\nonumber\\
&+\Big(\int |f_{\sigma,\delta}|^{\frac{6}{5}}\Big)^{\frac{5}{6}}
\Big(\int v_{\sigma,\delta}^6dg_{\sigma,\delta}\Big)^{\frac{1}{6}}.
\end{align}
The Sobolev inequality gives that
\[
\Big(\int v_{\sigma,\delta}^6dg_{\sigma,\delta}\Big)^{\frac{1}{3}}\leq C_{\sigma,\delta}\int|\nabla_{g_{\sigma,\delta}}v_{\sigma,\delta}|^2dg_{\sigma,\delta},
\]
where $C_{\sigma,\delta}$ denotes the Sobolev Constant of the metric $g_{\sigma,\delta}$. Then we have
\begin{align}
\Big(\int v_{\sigma,\delta}^6dg_{\sigma,\delta}\Big)^{\frac{1}{3}}\leq &
C_{\sigma,\delta}\Big(\int |f_{\sigma,\delta}|^{\frac{3}{2}}dg_{\sigma,\delta}\Big)^{\frac{2}{3}}
\Big(\int v_{\sigma,\delta}^6dg_{\sigma,\delta}\Big)^{\frac{1}{3}}\nonumber\\
&+\frac{1}{16}C_{\sigma,\delta}^2\Big(\int |f_{\sigma,\delta}|^{\frac{6}{5}}dg_{\sigma,\delta}\Big)^{\frac{5}{3}}
+\frac{1}{2}\Big(\int v_{\sigma,\delta}^6dg_{\sigma,\delta}\Big)^{\frac{1}{3}}.
\end{align}
Note $g_{\sigma,\delta}$ is uniformly close to $g_{\sigma}$.
Then $C_{\sigma,\delta}$ is uniformly bounded by corollary \ref{sob:inequ}. 
By (\ref{integral:estimate}), for $\sigma\gg1$ and $\delta<\frac{1}{8C\sigma^3}$, we have
\[
C_{\sigma,\delta}\Big(\int |f_{\sigma,\delta}|^{\frac{3}{2}}\Big)^{\frac{2}{3}}\leq\frac{1}{4}.
\]
Then it follows that for $\sigma\gg1$ and $\delta<\frac{1}{\sigma^3}$
\[
\Big(\int v_{\sigma,\delta}^6dg_{\sigma,\delta}\Big)^{\frac{1}{3}}
\leq C\Big(\int |f_{\sigma, \delta}|^{\frac{6}{5}}\Big)^{\frac{5}{3}}\leq
C(\sigma^{-\frac{6}{5}(\tau-\frac{1}{2})}+\sigma^2\delta)^{\frac{5}{3}}=o(1), \ \ \text{as} \ \sigma\rightarrow \infty.
\]
\textbf{Step 2}  We use Moser iteration to improve the estimate. We omit the lower index for simplicity.
By (\ref{eqn2:conformal}), we have
\[
\varphi^2 v^{2p-1}\Delta v=f\varphi^2v^{2p}+f\varphi^2v^{2p-1},
\]
where $\varphi$ is a $C^2$ function  supported in $B_{\frac{\sigma}{4}}(x)$ and $p\geq 3$ is positive integer.
Then Stokes' formula implies that
\[
-\int\varphi^2 v^{2p-1}\Delta v=(2p-1)\int \varphi^2v^{2p-2}|\nabla v|^2
+2\int \varphi v^{2p-1}\nabla\varphi\nabla v.
\]
It follows that
\begin{align}
&(2p-1)\int \varphi^2v^{2p-2}|\nabla v|^2\nonumber\\
=&-2\int \varphi v^{2p-1}\nabla\varphi\nabla v
-\int f\varphi^2v^{2p}-\int f\varphi^2v^{2p-1}\nonumber\\
\leq&\frac{2p-1}{2}\int \varphi^2v^{2p-2}|\nabla v|^2+\frac{2}{2p-1}\int|\nabla\varphi|^2v^{2p}
+\int |f|\varphi^2(v^{2p}+|v|^{2p-1})
\end{align}
On the other hand, using Sobolev inequality, we have
\begin{align}
\Big(\int(\varphi v^p)^6\Big)^{\frac{1}{3}}\leq& C\int|\nabla(\varphi v^p)|^2\nonumber\\
\leq &C\Big(\int|\nabla\varphi|^2v^{2p}+\int p^2v^{2p-2}|\nabla v|^2\Big)
\end{align}
Combining the two inequalities above gives
\begin{align}
\Big(\int(\varphi v^p)^6\Big)^{\frac{1}{3}}\leq&
C_1p^2\int|f|\varphi^2v^{2p}+C_1p^2\int|f|\varphi^2 |v|^{2p-1}+C_1\int|\nabla\varphi|^2v^{2p}.
\end{align}
By Holder inequality, 
\begin{align}
\int|f|\varphi^2 |v|^{2p-1}\leq&\Big(\int|f|\varphi^2v^{2p}\Big)^{\frac{2p-1}{2p}}
\Big(\int|f|\varphi^2\Big)^{\frac{1}{2p}}\nonumber\\
\leq&\sigma^{\tau}\int|f|\varphi^2v^{2p}+\sigma^{\tau-2p\tau}\int
|f|\varphi^2\nonumber\\
\leq&\int\sigma^{\tau}|f|\varphi^2v^{2p}+\sigma^{1-2\tau p}.
\end{align}
Then
\[
\Big(\int(\varphi v^p)^6\Big)^{\frac{1}{3}}\leq2 C_1 p^2\int\sigma^{\tau}|f|\varphi^2v^{2p}
+C_1\int|\nabla\varphi|^2v^{2p}+C_1p^2\sigma^{1-2\tau p}
\]
Using Holder inequality,
\begin{align}
p^2\int\sigma^{\tau}|f|\varphi^2v^{2p}\leq &
p^2\int_{\text{supp} \varphi}\Big(\sigma^{2\tau}|f|^2\Big)^{\frac{1}{2}}\Big(\int|\varphi v^p|^6\Big)^{\frac{1}{4}}
\Big(\int\varphi^2 |v|^{2p}\Big)^{\frac{1}{4}}\nonumber\\
\leq&\varepsilon \Big(\int|\varphi v^p|^6\Big)^{\frac{1}{3}}+\varepsilon^{-4}
p^8\Big(\int_{\text{supp}\varphi}\sigma^{2\tau}|f|^2\Big)^{2}\int\varphi^2v^{2p}
\end{align}
Take $\varepsilon$ such that $2C_1\varepsilon=\frac{1}{2}$. Then we have
\begin{align}\label{integral:estimate:last}
\Big(\int(\varphi v^p)^6\Big)^{\frac{1}{3}}\leq&
C_2p^8\Big(\int\varphi^2 \sigma^{2\tau}\Big|f|^2\Big)^{2}\int\varphi^2v^{2p}
+C_2\int|\nabla\varphi|^2v^{2p}+
C_2p^2\sigma^{1-2\tau p}\nonumber\\
\leq& C_3p^{8}\Big((\sigma^{-2}\int\varphi^2 v^{2p}
+\int|\nabla \varphi|^2v^{2p}+\sigma^{1-2\tau p}\Big).
\end{align}
We choose $\varphi_i\in C^2_c(B_{\frac{\sigma}{4}}(x))$ to be the  cut-off function depending only on the
distance to $x$ such that
\[
\varphi_i(x)=
\left\{
\begin{array}{lll}
1,\ \ \ \ \ \ \ \ \ \ \ \ \ x\in B_{r_{i+1}}(x),\\
0, \ \ \ \ \ \ \ \ \ \ \ \ \ x\notin B_{r_{i}}(x),
\end{array}
\right.
\]
and $|\nabla\varphi_i|\leq \frac{C2^{i}}{\sigma}$ for some uniform $C$. Here $r_i$ is defined by
\[
r_i=\frac{\sigma}{4}(1-\sum_{k=1}^i \frac{1}{2^{k+1}}).
\]
Set
\[
p=3^i\ \ \ \text{and} \ \ \ I_{i+1}=\sigma^{-3}\int_{B_{r_{i+1}}}|v|^{2\cdot3^{i+1}}+\sigma^{-2\tau3^{i+1}}.
\]
Then it follows (\ref{integral:estimate:last}) that
\begin{align}
I_{i+1}\leq&
C_3^33^{24i}\sigma^{-3}\Big[(\sigma^{-2}\int \varphi^2v^{2\cdot3^i}
+\int|\nabla \varphi_i |^2
v^{2\cdot3^i}+\sigma^{1-2\tau 3^i}\Big]^3+\sigma^{-2\tau3^{i+1}}\nonumber\\
\leq&C_43^{24i}8^i\Big[\sigma^{-9}(\int_{\text{supp} \varphi_i}v^{2\cdot3^i}\big)^3+\sigma^{-2\tau3^{i+1}}\Big]\nonumber\\
\leq&C_43^{24i}8^i\Big[\sigma^{-3}\int_{\text{supp} \varphi_i}v^{2\cdot3^i}+\sigma^{-2\tau3^{i}}\Big]^3
\nonumber\\
\leq&C_43^{24i}8^i I_i^3.
\end{align}
It's easy to show that
\[
I_i^{\frac{1}{2\cdot3^i}}\leq C_5 I_1^{\frac{1}{6}}.
\]
Sending $i$ to $\infty$ gives
\[
|v|(x)\leq C_5\Big(\sigma^{-\frac{1}{2}}(\int v^6)^{\frac{1}{6}}+\sigma^{-\tau}\Big)\leq C_5\sigma^{-\frac{1}{2}}.
\]
Hence, we finish the proof.
\end{proof}

\begin{thm}\label{mass:compare}
We can find some  $\sigma_0\gg1$  such that
for any $\sigma>\sigma_0$ and $\delta\leq\frac{1}{C\sigma^3}$, it holds
\[
m(\tilde{g}_{\sigma,\delta})\leq\frac{7}{8}m(g_{\sigma,\delta}).
\]
\end{thm}
\begin{proof}
Using the definition of mass, a straightforward calculation  yields
\[
m(\tilde{g}_{\sigma,\delta})=m(g_{\sigma,\delta})+2A_{\sigma,\delta}.
\]
where $A_{\sigma,\delta}$ is given by the expansion $u_{\sigma,\delta}(x) = 1 + \frac{A_{\sigma,\delta}}{|x|}+O(\frac{1}{|x|^2})$. Note that $f_{\sigma,\delta}=\frac{1}{8}R_{\sigma,\delta}$ outside 
$(\Sigma_{\sigma}\cup\Sigma_{\sigma}')\times[-\delta,\delta]$.
Applying integration by parts to (\ref{conformal:equ}) multiplied by $u_{\sigma,\delta}$, we have that
\begin{align}\label{mass:expansion}
4\pi A_{\sigma,\delta}
=&\int\Big(-f_{\sigma,\delta}u_{\sigma,\delta}^2
-|\nabla_{g_{\sigma,\delta}}u_{\sigma,\delta}|^2\Big)dg_{\sigma,\delta}\nonumber\\
\leq&2\int_{(\Sigma_{\sigma}\cup\Sigma_{\sigma}')\times[-\delta,\delta]}|f_{\sigma,\delta}|dg_{\sigma,\delta}
-\frac{1}{8}\int_{\Sigma_{\sigma}\times[\delta,\frac{\sigma}{2}-\delta]}R_{\sigma,\delta}
u_{\sigma,\delta}^2dg_{\sigma,\delta}\nonumber+2\int_{M'}|R_{\sigma, \delta}|dg_{\sigma,\delta}\nonumber\\
\leq& C\sigma^2\delta+2\int_{M'}|R_g|dg
-\frac{1}{8}\int_{\Sigma_{\sigma}\times[\delta,\frac{\sigma}{2}-\delta]}R_{\sigma,\delta}
u_{\sigma,\delta}^2dg_{\sigma,\delta},
\end{align}
where $M'$ consists of the points outside $\Sigma_{\sigma}\times\{-\delta\}$. 
Recall that
\[
g_{\sigma,\delta}=\gamma=h(t)+dt^2=f(t)\omega(t)+dt^2,\ \ \ 
(x,t)\in\Sigma_{\sigma}\times[\delta,\frac{\sigma}{2}-\delta].
\]
Then by  (\ref{scalar:curv1}),
\begin{align}\label{mass:lose}
&\int_{\Sigma_{\sigma}\times[\delta,\frac{\sigma}{2}
-\delta]}R_{\sigma,\delta}u_{\sigma,\delta}^2dg_{\sigma,\delta}\nonumber\\
=&\int_{\Sigma_{\sigma}\times[0,\frac{\sigma}{2}]}R_{\gamma}u_{\sigma,\delta}^2dg_{\gamma}+O(\sigma^2\delta)\nonumber\\
=&\int_{\Sigma_{\sigma}\times[0,\frac{\sigma}{2}]}R_{\gamma}\big(1+O(\sigma^{-\frac{1}{2}})\big)dg_{\gamma}+O(\sigma^2\delta)\nonumber\\
=&\int_0^{\frac{\sigma}{2}}\int_{\Sigma_{\sigma}\times\{t\}}
\Big[2K_{h(t)}-\frac{2}{f\sigma^2}-\frac{1}{4}|\dot{\omega}|_{\omega}^2\Big]
dA_{h(t)}dt+O(\sigma^2\delta)+O(\sigma^{\frac{1}{2}-\tau})\nonumber\\
=&4\pi\sigma-\frac{2}{\sigma^2}\int_0^{\frac{\sigma}{2}}\int_{\Sigma_{\sigma}\times\{t\}}dA_{\omega(t)}dt
+O(\sigma^2\delta)+O(\sigma^{\frac{1}{2}-\tau})\nonumber\\
=&4\pi \sigma-\frac{4\pi R^2_{\sigma}}{\sigma}+O(\sigma^2\delta)+O(\sigma^{\frac{1}{2}-\tau})\nonumber\\
\geq& 4\pi m(g_{\sigma,\delta})+O(\sigma^2\delta)+O(\sigma^{\frac{1}{2}-\tau}),
\end{align}
where we have used 
\[
Area(\Sigma_{\sigma}\times\{t\}, \omega(t))=4\pi R^2_{\sigma},\ \ \ \text{for}\ \ \ t\in[0,\frac{\sigma}{2}].
\]
As $\int_M|R|d_g$ is finite, we can choose $\sigma_0\gg1$ 
such that for any $\sigma\geq\sigma_0$ and
$\delta<\frac{1}{C\sigma^3}$,  
\[
C\sigma^2\delta+\int_{M'}|R|dg
+O(\sigma^2\delta)+O(\sigma^{\frac{1}{2}-\tau})\leq\frac{\pi m(g_{\sigma,\delta})}{4}.
\]
Combining above with (\ref{mass:expansion}) and (\ref{mass:lose}) yields the desired estimate.
\end{proof}

\section{Proof of Theorem \ref{mainthm}}
Now we turn to the proof of the main theorem.
Let $\{\Omega_{V_k}\}$ be isoperimetric regions of volumes $V_k \to \infty$
 and  $\{\Omega_k\}$ be the unique large component of $\{\Omega_{V_k}\}$ with $V(\Omega_k)=\frac{4\pi\rho_k^3}{3}\rightarrow\infty $. We know that $\Omega_k$ is connected.
Let $\tilde \Omega_k$ be the subset of $\{x \in \\R^3 : \rho_k \, |x| > 1/2\}$ such that 
\[
\Omega_k \setminus K \cong \{\rho_k  x : x \in \tilde \Omega_k\}
\]
Then upon passing to a subsequence, 
\[
\tilde \Omega_k \to B_1(\xi) \qquad \text{ in }  \qquad C^{2, \alpha}_{loc} (\R^3 \setminus \{0\})\ \ 
\text{for some}\ \ \xi \in \R^3.
\]
as $k \to \infty$. Our goal will be to show that $\xi = 0$. 
\begin{pro}\label{prop:centering-iso}
$\xi = 0$.
\end{pro}
\begin{re} An analogous result holds in all dimensions provided the sharp isoperimetric inequality holds for the conformal metric of smaller mass: Suppose the sharp isoperimetric inequality holds for any $n$-dimensional AF manifold $(M, g)$with nonnegative scalar curvature and positive mass.
If one can construct some scalar nonnegative metric  $\tilde{g}=u^{\frac{4}{n-2}}\bar{g}$ with 
$\bar{g}$ being some compact perturbation of $g$ 
and $u = 1 + A |x|^{2 - n} + O (|x|^{1- n})$ being some conformal deformation satisfying $A<0$,
then  large isoperimetric regions in AF manifolds with positive mass must be close to the corresponding centered coordinate balls.
\end{re}
\begin{proof}
Assume that $\xi \neq 0$.  By Theorem \ref {mass:compare}, we can choose some fixed $\sigma_0 \gg 1$
 and $\delta_0\leq \frac{1}{C\sigma_0^3}$ such that the conformal metric 
\[
\tilde{g}_{\sigma_0,\delta_0}=u^4_{\sigma_0,\delta_0}g_{\sigma_0,\delta_0}
=\Big(1+\frac{A_{\sigma_0,\delta_0}}{|x|}\Big)^4g_{\sigma_0,\delta_0}+O(|x|^{-2})
\]
satisfies
\begin{equation}
m(\tilde{g}_{\sigma_0,\delta_0})=m(g_{\sigma_0,\delta_0})+2A_{\sigma_0,\delta_0}
=m(g_{\sigma_0,\delta_0})(1-\varepsilon_0 ),
\end{equation}
for some $\varepsilon_0>0$.  Set
$\Omega_k'=\Omega_k\setminus \overline{B_{\frac{1}{2}}(0)} $. Then we have
\begin{align*}
V(\Omega'_k, \tilde{g}_{\sigma_0,\delta_0})=&V(\Omega_k, g_{\sigma_0,\delta_0})
-3\varepsilon_0m(g_{\sigma_0,\delta_0})\int_{\Omega'_k}\frac{1}{|x|}+o(\rho_k^2),\\
A(\partial{\Omega'_k}, \tilde{g}_{\sigma_0,\delta_0})=&A(\partial{\Omega}_k, g_{\sigma_0,\delta_0})
-2\varepsilon_0m(g_{\sigma_0,\delta_0})\int_{\partial\Omega'_k}\frac{1}{|x|}+o(\rho_k).
\end{align*}
On the other hand, $\{\Omega_k\}$ is a sequence of isoperimetric regions in $(M^3, g)$ and
$g=g_{\sigma_0,\delta_0}$ outside of some compact set.
Then
\[
V(\Omega_k, g_{\sigma_0,\delta_0})\geq
\frac{1}{6 \sqrt{\pi}}A^{\frac{3}{2}}(\partial{\Omega_k},g_{\sigma_0,\delta_0})
+\frac{m(g_{\sigma_0,\delta_0})}{2}A(\partial{\Omega_k},g_{\sigma_0,\delta_0})+o(\rho^2_k).
\]
Combining the above three inequalities yields 
\begin{align}
&V(\Omega'_k, \tilde{g}_{\sigma_0,\delta_0})-
\frac{1}{6 \sqrt{\pi}}A(\partial{\Omega}'_k, \tilde{g}_{\sigma_0,\delta_0})
-\frac{m(\tilde{g}_{\sigma_0,\delta_0})}{2}A(\partial{\Omega'_k}, \tilde{g}_{\sigma_0,\delta_0})\nonumber\\
\geq&V(\Omega_k, g_{\sigma_0,\delta_0})-
\frac{1}{6 \sqrt{\pi}}\Big(A(\partial{\Omega}_k, g_{\sigma_0,\delta_0})
-2\varepsilon_0m(g_{\sigma_0,\delta_0})
\int_{\partial\Omega'_k}\frac{1}{|x|}+o(\rho_k)\Big)^{\frac{3}{2}}\nonumber\\
&-\frac{m(g_{\sigma_0,\delta_0})}{2}A(\partial{\Omega_k}, g_{\sigma_0,\delta_0})
+\frac{\varepsilon_0 m(g_{\sigma_0,\delta_0})}{2}A(\partial{\Omega'_k},g_{\sigma_0,\delta_0})\nonumber\\
&-3\varepsilon_0m(g_{\sigma_0,\delta_0})\int_{\Omega'_k}\frac{1}{|x|}+o(\rho_k^2)\nonumber\\
\geq& V(\Omega_k, g_{\sigma_0,\delta_0})-
\frac{1}{6 \sqrt{\pi}}A^{\frac{3}{2}}(\partial{\Omega_k},g_{\sigma_0,\delta_0})
-\frac{m(g_{\sigma_0,\delta_0})}{2}A(\partial{\Omega_k},g_{\sigma_0,\delta_0})\nonumber\\
&+\varepsilon_0 m(g_{\sigma_0,\delta_0})\Big(2\pi\rho_k^2
+\rho_k\int_{\partial\Omega'_k}\frac{1}{|x|}
-3\int_{\Omega'_k}\frac{1}{|x|}\Big)+o(\rho_k^2)\nonumber\\
\geq& \varepsilon_0 m(g_{\sigma_0,\delta_0})\Big(2\pi\rho_k^2
+\rho_k\int_{\partial\Omega'_k}\frac{1}{|x|}-
3\int_{\Omega'_k}\frac{1}{|x|}\Big)+o(\rho_k^2)
\end{align}
Note that
 $\tilde \Omega_k \to B_1(\xi)\ \text{ in }\   C^{2, \alpha}_{loc} (\R^3 \setminus \{0\})$.
 Then 
\begin{align}
\rho_k\int_{\partial\Omega'_k}\frac{1}{|x|}=\rho_k\int_{S_{\rho_k}(\rho_k\xi)}\frac{1}{|x|}+o(\rho^2_k)
=\begin{cases}
4\pi\rho_k^2+o(\rho^2_k) & |\xi| \leq1\\
\frac{4\pi\rho^2_k}{|\xi|}+o(\rho^2_k) & |\xi| \geq 1.
\end{cases}
\end{align}
Similarly, 
\begin{align}
\int_{\Omega'_k}\frac{1}{|x|}=\int_{B_{\rho_k}(\rho_k\xi)}\frac{1}{|x|}+o(\rho^2_k)
=\begin{cases}
2\pi\rho_k^2(1-\frac{|\xi|^2}{3})+o(\rho^2_k) & |\xi| \leq1\\
\frac{4\pi\rho^2_k}{3|\xi|}+o(\rho^2_k) & |\xi| \geq 1.
\end{cases}
\end{align}
Hence, we always have
\begin{align}
&V(\Omega'_k, \tilde{g}_{\sigma_0,\delta_0})-
\frac{1}{6 \sqrt{\pi}}A(\partial{\Omega}'_k, \tilde{g}_{\sigma_0,\delta_0})
-\frac{m(\tilde{g}_{\sigma_0,\delta_0})}{2}A(\partial{\Omega'_k}, \tilde{g}_{\sigma_0,\delta_0})\nonumber\\
\geq&\frac{2\varepsilon_0 \pi m(g_{\sigma_0, \delta_0})|\xi|^2\rho^2_k}{1+|\xi|^2}+o(\rho^2_k),
\end{align}
which contradicts with the sharp isoperimetric inequality 
 \eqref{eqn:isoperimetricinequality} on $(\R^3, \tilde{g}_{\sigma_0,\delta_0})$. 
\end{proof}
\begin {proof} [Proof of Theorem \ref{mainthm}] 
By Proposition \ref{prop:centering-iso}, we see that every sufficiently large isoperimetric region
 is connected and close to the centered coordinate ball $B_1(0)$ 
after suitable scaling  in the chart at infinity (\ref{chart:infinity}). 
The uniqueness of large stable constant mean curvature spheres obtained by Nerz in \cite{Nerz2015} 
shows  the  boundary of such an isoperimetric region must be a leaf of the canonical foliation. 
\end{proof}

\end{document}